\documentclass[12pt]{amsart}
\usepackage{amssymb}
\usepackage{amscd}
\usepackage{amsthm}
\usepackage{enumerate}
\usepackage[dvips]{graphicx}
\usepackage[all,cmtip]{xy}
\usepackage{hyperref}

\usepackage{color}

\usepackage[all]{xy}

\newtheorem{Thm}{Theorem}[section]
\newtheorem{Lem}[Thm]{Lemma}
\newtheorem{Def}[Thm]{Definition}
\newtheorem{Cor}[Thm]{Corollary}
\newtheorem{Prop}[Thm]{Proposition}
\newtheorem{Ex1}[Thm]{Example}
\newtheorem{Rem1}[Thm]{Remark}

\newtheorem{Assumption}[Thm]{Assumption}

\newenvironment{Rem}{\begin{Rem1}\rm}{\end{Rem1}}
\newenvironment{Ex}{\begin{Ex1}\rm}{\end{Ex1}}

\numberwithin{equation}{section}

\title{On simple-minded systems and $\tau$-periodic modules of self-injective algebras}
\author{Aaron Chan, Yuming Liu and Zhen Zhang}

\address{Aaron Chan
\newline Graduate School of Mathematics
\newline Nagoya University
\newline Furocho, Chikusaku
\newline Nagoya, Japan}
\email{aaron.kychan@gmail.com}

\address{Yuming Liu
\newline School of Mathematical Sciences
\newline Laboratory of Mathematics and Complex Systems
\newline Beijing Normal University
\newline Beijing 100875
\newline P.R.China}
\email{ymliu@bnu.edu.cn}

\address{Zhen Zhang
\newline School of Mathematical Sciences
\newline Beijing Normal University
\newline Beijing 100875
\newline P.R.China}
\email{zhangzhen@mail.bnu.edu.cn}

\date{version of \today}

\newenvironment{Proof}[1][Proof]{\begin{trivlist}
\item[\hskip \labelsep {\bfseries #1}]}{\flushright
$\Box$\end{trivlist}}

\newcommand{\lra}{\longrightarrow}

\newcommand{\ra}{\rightarrow}
\newcommand{\sdp}{\times\kern-.2em\vrule height1.1ex depth-.05ex}
\newcommand{\epi}{\lra \kern-.8em\ra}

\newcommand{\ul}{\underline}

\newcommand{\stmod}{\underline{\mathrm{mod}}}
\newcommand{\sthom}{\underline{\mathrm{Hom}}}
\newcommand{\stend}{\underline{\mathrm{End}}}

 \normalbaselines
\thanks{We would like to thank Steffen Koenig for his comments and suggestions on a preliminary version of this paper. We would like to thank Andrzej Skowronski for some correspondence on the results in the paper \cite{MS}. This research was started when the first author visit Beijing Normal University, and he would like to thank the School of Mathematical Sciences for their hospitality. The second author is supported by NCET Program from MOE of China and by NNSF (No.11171325, No.11331006).}

\begin{document}
\renewcommand{\thefootnote}{\alph{footnote}}
\renewcommand{\thefootnote}{\alph{footnote}}
\setcounter{footnote}{-1} 
\renewcommand{\thefootnote}{\alph{footnote}}
\setcounter{footnote}{-1} 

\begin{abstract}
Let $A$ be a finite-dimensional self-injective algebra over an algebraically closed field, $\mathcal{C}$ a stably quasi-serial component (i.e. its stable part is a tube) of rank $n$ of the Auslander-Reiten quiver of $A$, and $\mathcal{S}$ be a simple-minded system of the stable module category $\stmod{A}$. We show that the intersection $\mathcal{S}\cap\mathcal{C}$ is of size strictly less than $n$, and consists only of modules with quasi-length strictly less than $n$.
In particular, all modules in the homogeneous tubes of the Auslander-Reiten quiver of $A$ cannot be in any simple-minded system.
\end{abstract}

\maketitle

\section{Statements of the main results}

\medskip
Let $k$ be a commutative artin ring, and $A$ an artin $k$-algebra.
We denote by mod$A$ the category of all finitely generated left
$A$-modules, and by
$\stmod A$ the stable category of mod$A$, that is, the category with the same class of objects but with morphism spaces $\sthom_A(X,Y)$ being the quotient of the ordinary one $\mathrm{Hom}_A(X,Y)$ by maps factoring through projective modules.
{\it Although most definitions and the problem we consider in this paper can be discussed in this more general setting, we only concentrate, for technical reasons, on the case when $k$ is an algebraically closed field and $A$ is a finite-dimensional self-injective $k$-algebra.}
We also further assume, without loss of generality, that $A$ is ring-indecomposable and non-simple throughout the article.

\medskip

\begin{Def}\label{stbrick}
An $A$-module $M$ is a \emph{stable brick} if $\stend_A(M)\cong k$.

A set $\mathcal{S}$ of $A$-modules is a \emph{stable semibrick} if it consists of pairwise orthogonal stable bricks, i.e. every $X\in \mathcal{S}$ is a stable brick and $\sthom_A(X,Y)=0$ for all distinct $X,Y\in \mathcal{S}$.
\end{Def}

Consider the following way of reconstructing objects of $\stmod{A}$.
Let $\mathcal{S}$ be a class of $A$-modules.
The full subcategory
$\langle\mathcal{S}\rangle$ of $\stmod{A}$ is the additive closure of
$\mathcal{S}$. Denote by
$\langle\mathcal{S}\rangle\ast\langle\mathcal{S'}\rangle$ the class
of indecomposable $A$-modules $Y$ such that there is a short exact
sequence $0 \rightarrow X \rightarrow Y\oplus P \rightarrow Z\rightarrow 0$ with $X\in \langle\mathcal{S}\rangle, Z\in\langle\mathcal{S'}\rangle$,
and $P$ projective. Define $\langle\mathcal{S}\rangle_1:=\langle\mathcal{S}\rangle$ and
$\langle\mathcal{S}\rangle_n:=\langle\langle\mathcal{S}\rangle_{n-1}\ast\langle\mathcal{S}\rangle\rangle$ for $n>1$.

Under our assumption on $A$, we can simplify the original definition (from \cite{KL}) of the main subject of interest as follows.

\begin{Def}\label{sms} {\rm(\cite{KL})}
Let $A$ be a self-injective algebra over an algebraically closed field.  A set $\mathcal{S}$ of objects in $\stmod A$ is called a \emph{simple-minded system} (or \emph{sms} for short)
if the following conditions are satisfied:
\begin{enumerate}
\item (orthogonality) $\mathcal{S}$ is a stable semibrick.
\item (finite filtration) For each indecomposable non-projective
$A$-module $X$, there exists some natural number $n$ (depending on
$X$) such that $X\in\langle\mathcal{S}\rangle_n$.
\end{enumerate}
\end{Def}

It is clear that the set of (isoclass representatives of) simple modules is an example of sms.
It has been shown in \cite{KL} that each sms {has} finite
cardinality and the sms's are invariant under stable equivalences.  One of the fundamental questions concerning sms's is the following ``simple-image problem".  Namely, given an sms $\mathcal{S}$ of $A$, is this the image of the simple modules under some stable equivalence \underline{mod}$B\rightarrow$ \underline{mod}$A$? It is shown in \cite{CKL} that the answer to this question is true for representation-finite self-injective algebras.

\medskip

As a generalisation of the notion of simple modules, we are interested in finding how far one can generalise various properties of simple modules to that of sms's.  In this note, we prove one of such properties.  Let us be more specific now.

Recall that the {\it Auslander-Reiten quiver} (AR-quiver) $\Gamma_A$ of A is a translation quiver
whose vertices are the isomorphism classes of indecomposable (finitely generated)
$A$-modules, arrows are the irreducible maps valued by their multiplicities, and
whose translation is the Auslander-Reiten translate $\tau$ (see \cite{ARS}). A (connected) component
$\mathcal{C}$ of $\Gamma_A$ is a {\it homogeneous tube} if it is of the form $\mathbb{Z}A_\infty /\langle\tau\rangle$ (see \cite{Ringel}).  In particular, all modules in a homogeneous tube of $\Gamma_A$ are of $\tau$-period $1$.
Note that none of the simple modules of a self-injective algebra lie in a homogeneous tube.

More generally, following Erdmann and Kerner \cite{EK}, we call a component
$\mathcal{C}$ of $\Gamma_A$ {\it stably quasi-serial of rank $n$} if its stable part (that is, the full subquiver obtained by removing all vertices corresponding to indecomposable projective-injective modules) is of the form $\mathbb{Z}A_\infty /\langle\tau^n\rangle$.
It is known that a stably quasi-serial component of rank $n$ contains at most $n-1$ simple modules (see \cite{MS}).
Our first main result generalizes this result to any sms.

\begin{Thm}\label{main-result-1}
Let $A$ be a self-injective algebra over an algebraically closed field and $\mathcal{C}$ a stably quasi-serial component of rank $n$. Then the number of elements in an sms of $A$ lying in $\mathcal{C}$ is strictly less than $n$. In particular, none of the indecomposable module in an sms of $A$ lie in the homogeneous tubes of the AR-quiver $\Gamma_A$ of $A$.
\end{Thm}

\medskip
According to a result of Crawley-Boevey \cite{CB}, if $A$ is tame, then ``almost all modules" (or precisely, for each $d>0$, all but a finite number of isomorphism classes of indecomposable $A$-modules of dimension $d$) lie in homogeneous tubes.  Therefore, our result excludes most of the modules of a tame self-injective algebra from forming an sms.

\medskip
\begin{Rem}\label{rem:bdd}
We remark that the bound given in Theorem \ref{main-result-1} is the best possible.
For a concrete example, for each $n>1$, consider the algebra $A=kQ/I$ whose quiver $Q$ is given by
$$ Q: \xymatrix{
 1 \ar[d]_{\alpha} & & n+1 \ar[ll]_{\alpha} \ar@/^/@<2ex>[d]^{\gamma} \\
 2 \ar[r]_{\alpha} & \cdots \ar[r]_{\alpha} & n \ar@/^/[u]^{\alpha}  \ar@/_/@<-1ex>[u]^{\gamma}
}
$$
and $I$ is generated by $\alpha\gamma$, $\gamma\alpha$, and $\alpha^{n+1}-\gamma^{2}$, whenever these compositions make sense.
Then $A$ is a (tame) symmetric special biserial algebra and one can calculate its AR-quiver explicitly; see, for example, \cite{Erdmann}.
$A$ has two stably quasi-serial component of rank $n$.
The simple modules $S_i$ for $i\in\{1,2,\ldots, n-1\}$ all lie in the same stably quasi-serial component of rank $n$; in fact, they all lie on the mouth of this component.
Note that $S_n$ and $S_{n+1}$ lie in distinct Euclidean components of the stable AR-quiver.
\end{Rem}

\medskip
While it is possible to have modules in a stably quasi-serial component of higher rank belonging to an sms, one can deduce easily from results in \cite{EK} (see Lemma \ref{r<n}) that the quasi-lengths (that is, the positions in an infinite sectional path $A_\infty = (1\to 2 \to 3 \to \cdots) \subset \mathbb{Z}A_\infty/\langle\tau^n\rangle$ in the stable part of the component) of such modules are not more than the rank $n$. Our second main result shows that the quasi-lengths of such modules are always less than the rank.

\begin{Thm}\label{main-result-2}
Let $A$ be a self-injective algebra over an algebraically closed field and $\mathcal{C}$ a stably quasi-serial component of rank $n$. Then every object in an sms lying in $\mathcal{C}$ has quasi-length less than $n$.
\end{Thm}

\begin{Rem} We note that a module of quasi-length $l$ for any $1\leq l<n$ in a stably quasi-serial component of rank $n$ can be in an sms of $\stmod{A}$.
For a concrete example, consider the path algebra $A=KQ/I$ whose quiver $Q$ is given by
$$ Q: \xymatrix{
  1 \ar@/^/[r]^{\alpha} & 2 \ar@/^/[l]^{\beta}   \ar@/^/[r]^{\alpha} &\cdots  \ar@/^/[l]^{\beta} \ar@/^/[r]^{\alpha} & n-1  \ar@/^/[l]^{\beta} \ar[dr]_{\gamma} & n+1  \ar[l]_{\gamma}  \ar@/^/@<2ex>[d]^{\delta} \\
 & &  &  & n \ar[u]^{\gamma}  \ar@/_/@<-1ex>[u]^{\delta}  }
 $$ 
and $I$ is generated by
$$\alpha^{2}, \beta^{2}, \alpha\gamma, \gamma\beta, \delta\gamma, \gamma\delta, \alpha\beta-\beta\alpha, \beta\alpha-\gamma^{3}, \delta^{2}-\gamma^{3},$$
whenever these compositions make sense.
Then $A$ is a (tame) symmetric special biserial algebra; in fact, it is derived (hence, stably) equivalent to the algebra in Remark \ref{rem:bdd}.
The simple module $S_{n-1}$ corresponding to the vertex $n-1$ is of quasi-length $n-1$ in one of the rank-$n$ tubes.
In fact, if $n>1$, then we have $S_1=X(1)$ and $S_2=Y(2)$ are of quasi-length $1$ and $2$ respectively lying in distinct rank $n$ tubes, and for $i\leq n-1$, we have $S_{i}=\tau^{j}(X(i))$ for all odd $i=2j+1$ and $S_{i}=\tau^{j-1}(Y(i))$ for all even $i=2j$.  Here $S_i$ is a module of quasi-length $i$.
\end{Rem}
\medskip

This article is structured as follows.
In the first subsection of Section \ref{sec:strat-cor}, we will explain the core strategy in the proof of the two main theorems, and give some brief comments of our results in the more general setting of Hom-finite Krull-Schmidt triangulated categories.
In the second subsection, we will present some easy consequences of the main theorems.

The remaining part will then be devoted to proving the main theorems.
We will recall some known results of the Auslander-Reiten theory of finite dimensional self-injective algebras from \cite{EK} in Section \ref{sec:remind}, for which our proofs rely heavily on.
The final two sections are devoted solely to proving Theorem \ref{main-result-1} and Theorem \ref{main-result-2} respectively.

\bigskip

\section{Strategy and Consequences}\label{sec:strat-cor}
From now on, $A$ will always be a ring-indecomposable non-simple self-injective algebra.
Recall that the stable category of a self-injective algebra has a triangulated structure, with suspension functor being the cosyzygy functor $\Omega^{-1}$ (see, for example, \cite{Happel}).  We will freely use the properties of this triangulated structure.
We will often use $\ul{(X,Y)}$ in place of $\sthom_A(X,Y)$ for a cleaner presentations of various exact sequences.
For $f:X\to Y$ an $A$-module map, we denote by $\underline{f}$ the image in the quotient $\sthom_A(X,Y)$.

\subsection{Strategy of proof}\label{subsec:strat}

Let us first remark that the proof we use is completely different from that of \cite{MS}.  The method in \cite{MS} relies on looking at the (composition) length of the projective modules in the component $\Omega(\mathcal{C})$; whereas our method is, roughly speaking, to show that, when an sms $\mathcal{S}$ contains undesired modules in $\mathcal{C}$, then there are modules have infinite ``length with respect to $\mathcal{S}$", which is a contradiction to the definition of sms.
Let us be more precise about this.

In \cite{Dugas}, Dugas defines sms's in a more general setting: Hom-finite Krull-Schmidt triangulated $k$-category $\mathcal{T}$.
The definition he uses is {\it not} just replacing the short exact sequences used to define $\langle\mathcal{S}\rangle_n$ by triangles, but the following.

Define for any two classes of objects $\mathcal{S}, \mathcal{S'}$ in $\stmod A$
$$\mathcal{S}\ast_\triangle\mathcal{S'}:=\{Y\in \mathcal{T} \mid\mbox{ there is a triangle }X \rightarrow Y \rightarrow Z \mbox{ with }X\in \mathcal{S}, Z\in\mathcal{S'}\}.$$

Then we can define $(\mathcal{S})_0:=\{0\}$ and
$(\mathcal{S})_n:=(\mathcal{S})_{n-1}\ast_\triangle(\mathcal{S}\cup{\{0\}})$ for $n\geq1$.
Now, a set $\mathcal{S}$ of $A$-modules in $\stmod{A}$, with $A$ self-injective and over an algebraically closed field, is an sms (in the sense of Dugas \cite{Dugas}) if $\mathcal{S}$ is a stable semibrick such that for all $X\in \stmod{A}$, we have $X\in (\mathcal{S})_n$ for some natural number $n$.

It is not difficult to see that this definition is equivalent to the one in Definition \ref{sms}.
Indeed, since $(\mathcal{S})_n$ is closed under direct summand when $\mathcal{S}$ is a stable semibrick \cite[Lemma 2.7]{Dugas}, so one can show that, by induction on $n$, $X\in \langle\mathcal{S}\rangle_n$ if, and only if, there is some $m\geq n$ such that $X\in (\mathcal{S})_{m}$.

The main difference of Dugas' definition and the original one in \cite{KL} is equivalent to the difference of filtering a module by its Loewy layer and by its composition factors. Indeed, for a stable semibrick $\mathcal{S}$ and a non-projective indecomposable module $X$, we can define
\[
\mathrm{LL}_\mathcal{S}(X) := \inf\{ n\geq 1 \mid X\in\langle\mathcal{S}\rangle_n \}, \;\;\text{and}\;\;
\ell_\mathcal{S}(X) := \inf\{ n\geq 1\mid X\in(\mathcal{S})_n \}.
\]
Then these should be viewed, respectively, as the Loewy length and composition length of $X$ with respect to $\mathcal{S}$ - it is clear that when $\mathcal{S}$ is the set of simple modules, then this coincides with the classical notion.
Now, Koenig-Liu's (resp. Dugas') definitions of sms's can be rephrased as a stable semibrick $\mathcal{S}$ such that every $X\in \stmod{A}$ has $\mathrm{LL}_{\mathcal{S}}(X)<\infty$ (resp. $\ell_{\mathcal{S}}(X)<\infty$).

The strategy that we will use to prove both Theorem \ref{main-result-1} and \ref{main-result-2} can now be more precisely stated as follows.

\begin{Prop}\label{main-strat}
Let $\mathcal{S}$ be a stable semibrick.
If there is a sequence $(M_i)_{i\geq 1}$ of pairwise non-isomorphic indecomposable non-projective $A$-module such that for all $i\geq 1$, we have
\begin{enumerate}[(i)]
\item $M_i\ncong S$ for all $S\in \mathcal{S}$;
\item for any $S\in \mathcal{S}$ and any non-split triangle $N\to M_i \to S\to $, every indecomposable direct summand of $N$ is isomorphic (in $\stmod{A}$) to $M_j$ for some $j>i$,
\end{enumerate}
then $\ell_{\mathcal{S}}(M_i)=\infty$ for all $i\geq 1$.
In particular, $\mathcal{S}$ cannot be an sms of $\stmod{A}$.
\end{Prop}
\begin{Proof}
Suppose the contrary.
Then there is some $i\geq 1$ with $M_i\in (\mathcal{S})_r\setminus (\mathcal{S})_{r-1}$ for some positive integer $r$.
In other words, we have a filtration of $M_i$ by elements of $\mathcal{S}$:
\[
\xymatrix@R=10pt{
S_r \ar[rr] & & N_{r-1} \ar[ld] \ar[rr] & & N_{r-2}\ar[ld]\ar@{.}[rr] & & N_2\ar[rr] && M_i, \ar[ld] \\
 & S_{r-1} \ar@{+->}[lu] & & S_{r-2}\ar@{+->}[lu]&&&& S_1\ar@{+->}[lu] &
}
\]
where $ S_j\in \mathcal{S}$ for all $j$ and every triangle in the picture is a triangle in $\stmod{A}$ (with $\xymatrix@C=15pt@1{X\ar@{+->}[r] & Y}$ denote the connecting morphism $X\to \Omega^{-1}(Y)$ in a triangle).
It follows from (ii) that every indecomposable direct summand of $N_2$ is of the form $M_j$ with $j>i$, and so all of these $M_j$'s are in $(\mathcal{S})_{r-1}$.
Repeat this argument down the filtration, we have that $S_r\cong M_j$ for some $j>i$, which contradicts (i).
\end{Proof}

When the conditions of the proposition hold, then by the equivalence of Koenig-Liu's and Dugas' definitions of sms, we also have $\mathrm{LL}_{\mathcal{S}}(X)=\infty$.  Indeed, if one prefers working in Koenig-Liu's setting instead, then one can modify (ii) to have $S\in \langle\mathcal{S}\rangle$ and allows the triangle there to contain direct summands given by ``trivial triangles" $\Omega^{-1}(S')\to0\to S'\to$.  Since this is slightly more fiddly to work with, in this article, we will use the version presented above.

Finally, we remark that the investigation carried out in this note can also be discussed in the setting of Hom-finite Krull-Schmidt $k$-linear (with $k$ algebraically closed) triangulated category $\mathsf{T}$ that admits Auslander-Reiten theory (hence, Serre duality), i.e.

\begin{itemize}
\item Does an sms of $\mathsf{T}$ always contain less than $n$ objects in a rank $n$ tube $\mathbb{Z}A_\infty/\langle \tau^n\rangle$?
\item Is the quasi-length of an object in a tube always strictly less than the rank?
\end{itemize}

While our argument can be applied in such a setting, the caveat is that one needs to check that all the results we took from \cite{EK} hold in the above setting - a careful reader can soon see that, as the proofs of these results come down to the defining property of almost split sequences, there is no danger to transfer all the arguments and results to the above setting.
Since sms was introduced as a mean of tackling Auslander-Reiten conjecture on stable equivalences, c.f. \cite{KL}, we chose to focus this article in the setting of stable module categories.

\medskip
\subsection{Consequence of the main results}
The first consequence of our results is on the relation between three different notions of simple-like generators.
Let us start by recalling the following one from \cite{KL}.

\begin{Def}
\label{wsms} {\rm(\cite{KL})}
Let $A$ be a self-injective algebra over an algebraically closed field.
A set $\mathcal{S}$ of $A$-modules is said to \emph{spans} $\stmod{A}$ if for all indecomposable non-projective $A$-module $X$, $\sthom_A(X,S)\neq 0$.
A \emph{weakly simple-minded system} (or \emph{wsms} for short) of $\stmod{A}$ is a stable semibrick which spans $\stmod{A}$.
\end{Def}

It has been shown \cite[Theorem 5.6]{KL} that when $A$ is representation-finite self-injective, wsms is sufficient for (hence, equivalent to) defining an sms.  In fact, the proof of this result is closely related to Proposition \ref{main-strat}, namely, any indecomposable $A$-module has finite Loewy length with respect to a wsms $\mathcal{S}$; hence, it must also be an sms of $\stmod{A}$.

In this respect, the proofs of Theorem \ref{main-result-1} and \ref{main-result-2} reflects how difficult (if not impossible) it is to modify \cite[Theorem 5.6]{KL} to the representation-infinite case.
Having said that, we do not know any example of wsms that is not an sms.

In \cite{Pogorzaly1994}, Pogorzaly investigated another candidate, called {\it maximal system of orthogonal stable bricks}, of simple-like generators of stable module categories.
Simply put, such a system is just a wsms with an extra condition: every indecomposable module $X$ in the system must satisfy $\tau X\ncong X$.

While it is easy to see that any sms is an wsms, there is no apparent relations between sms's and maximal systems of orthogonal stable bricks, that is, it is not clear if every indecomposable module $X$ in an sms must satisfies $\tau X\ncong X$.
Thanks to Theorem \ref{main-result-1}, we can now see that the implication from sms to wsms actually ‘‘factors through" maximal system of orthogonal stable bricks, for (almost) all self-injective algebras.

\begin{Cor}\label{coro} Let $A$ be a self-injective algebra over an algebraically closed field which is not a local Nakayama algebra. Any sms of $A$ is a maximal system of orthogonal stable bricks.
\end{Cor}
\begin{proof}
We have already mentioned that modules in a homogeneous tube are $\tau$-periodic of period $1$.  Here we only need an almost converse of this proved by Hoshino in \cite[Theorem 1]{Hoshino}.
 His result asserts that for a basic (indecomposable) artin algebra $\Lambda$ over algebraically closed field,
 if there is an (indecomposable) module $M$ with $\tau M\cong M$, then either $\Lambda$ is a local Nakayama algebra, or $M$ lies in a homogeneous tube.
 The claim now follows from applying this result to Theorem \ref{main-result-1}.
\end{proof}

We list two other immediate consequences of our main results.
The following one can be regarded as a generalisation of the property of simple modules being string modules for special biseral algebras.

\begin{Cor}\label{coro2} Let $A$ be a self-injective special biserial algebra over an algebraically closed field. Then no band $A$-module can be in an sms of $A$.
\end{Cor}
\begin{proof}
It is well-known that any indecomposable module (possibly except for a few indecomposable projective modules) over a special biserial algebra can be described either as a string module or as a band module, and that band modules lie in homogeneous tubes of $\Gamma_A$ (see \cite[Chapter II]{Erdmann}). The result now follows from Theorem \ref{main-result-1}.
\end{proof}

The following consequence of Theorem \ref{main-result-2} seems new.

\begin{Cor}\label{coro3} Let $A$ be a self-injective algebra over an algebraically closed field and $\mathcal{C}$ a stably quasi-serial component of rank $n$. Then any simple module lying in $\mathcal{C}$ has quasi-length less than $n$.
\end{Cor}

\medskip

\section{Reminders on stably quasi-serial component}\label{sec:remind}
\medskip

For general properties of stable categories for self-injective algebras and of Auslander-Reiten theory we refer to \cite{ARS,Erdmann,Ringel}.

In this section, we recall various notations and results from \cite[Section 2]{EK}, as well as some of their elementary implications that are not stated explicitly in \cite{EK}.

Let us fix a stably quasi-serial component $\mathcal{C}$ of rank $n\geq 1$ of the AR-quiver of $A$, i.e. removing projective modules in $\mathcal{C}$ yields a stable translation quiver $\mathbb{Z}A_\infty/\langle\tau^n\rangle$.

Following \cite{EK}, we can specify modules in $\mathcal{C}$ as follows.
Recall that a sectional path is a path $\cdots M_i \to M_{i+1} \to M_{i+2} \to \cdots$ in the AR-quiver such that $M_i\ncong \tau(M_{i+2})$.
If $X$ is an indecomposable non-projective module lying at the end of $\mathcal{C}$, that is, a {\it quasi-simple} of $\mathcal{C}$,
then for any natural number $r\geq1$, there is a unique infinite sectional path starting at $X$
\[
X = X(1)\rightarrow X(2)\rightarrow \cdots \rightarrow X(r)\rightarrow X(r + 1)\rightarrow \cdots;
\]
dually, there is a unique infinite sectional path in $\mathcal{C}$ ending at $X$:
\[\cdots \rightarrow [r+1]X \rightarrow [r]X \rightarrow \cdots \rightarrow [2]X \rightarrow [1]X=X.
\]
A non-projective module in $\mathcal{C}$ is {\it of quasi-length $r$} if it is of the form $X(r)$ for some quasi-simple $X$ of $\mathcal{C}$.
For notational convenience, we treat $X(0)=[0]X=0$.  Note that if $\mathcal{C}$ is homogeneous, then $X(r)=[r]X$ for all $r\geq 1$; otherwise, $X(r)\cong [r](\tau^{-(r-1)}X)$.

For any integer $i$, we denote by $\bar{i}$ the positive integer in $\{1,\ldots,n\}$ with $i\equiv\bar{i}$ mod $n$.
For convenience, we fix the notation $\{X_{i}\mid i=1, \ldots, n\}$ with $\tau X_{i}=X_{\overline{i-1}}$ for $i\in\{1,\ldots,n\}$ as the set of (non-projective) quasi-simple modules in $\mathcal{C}$.
By \cite[Proposition 2.3, Lemma 2.3.1]{EK}, there is a triangle in $\stmod{A}$ of the form
\begin{equation}\label{ses1}
 X_{i}(l)\xrightarrow{\underline{\epsilon}} X_{i}(l+j) \xrightarrow{\underline{\pi}} \tau^{-l} (X_{i}(j))\rightarrow \Omega^{-1}(X_{i}(l))
\end{equation}
for any $l,j>0, i\in\{1,\ldots,n\}$, where $\epsilon$ (resp. $\pi$) is given by the composition of the irreducible maps on the sectional path starting from $X_i(l)$ (resp. $X_i(l+j)$) and ending at $X_i(l+j)$ (resp. $\tau^{-l}(X_i(j))$).
By the following lemma, both $\underline{\epsilon}$ and $\underline{\pi}$ are non-zero in $\stmod{A}$.

\begin{Lem} $($\cite[Lemma 2.6]{EK}$)$ \label{sectional-path-mor}
The composition $\pi\colon [r]X_i \to [s]X_i$ ($\epsilon\colon \tau^{r-s}[s]X_i \to [r]X_i,$ respectively) of the chain of irreducible maps in a sectional path for $r>s\geq1$ does not factor through a projective module.
In particular, $\sthom_A([r]X_i, [s]X_i)$ and $\sthom_A(\tau^{r-s}[s]X_i, [r]X_i)$ are both non-zero.
\end{Lem}

\begin{Rem} \label{dually}
Take $X_j$ to be the quasi-simple with $X_j(r)=[r]X_i$, then the above compositions can also be presented as $\epsilon\colon X_j(s) \to X_j(r)$ and $\pi\colon X_j(r) \to \tau^{s-r}X_j(s)$ respectively.
In particular, we have $\sthom_A(X_j(s), X_j(r))\neq 0$ for any $1\leq s<r$ and any $j$.
\end{Rem}

We have the following generalisation of the triangle \eqref{ses1}.

\begin{Lem} \label{support}
Consider $X_{i}(r)$ with $r\geq1$ and $i\in\{1, 2, \ldots, n\}$. Then there are non-split triangles
\begin{equation}\label{ses2}
 X_{i}(r)\xrightarrow{\begin{scriptsize}\begin{pmatrix}\underline{\epsilon_{1}}\\ \underline{\pi_{1}}  \end{pmatrix}\end{scriptsize}} X_{i}(r+j)\oplus \tau^{-l}(X_{i}(r-l)) \xrightarrow{\begin{scriptsize}\begin{pmatrix}\underline{\pi_{2}},\underline{\epsilon_{2}} \end{pmatrix}\end{scriptsize}} \tau^{-l} (X_{i}(r-l+j))\rightarrow \Omega^{-1}(X_{i}(r))
\end{equation}
with $1\leq l\leq r$ and $j\geq 1$, and $\epsilon_{m},\pi_{m}$ for $m\in\{1,2\}$ are the compositions of irreducible maps in the sectional paths.
\end{Lem}
\begin{Proof}
This is a straightforward generalisation of the proof in \cite[Lemma 2.3]{EK}.
\end{Proof}
Note that the triangle (\ref{ses1}) is the special case of (\ref{ses2}) for $l=r$.

\begin{Lem}\label{stbrick-induces-hom}
For $r\geq1$ and $i\in\{1, 2,\ldots, n\}$, if $\stend_A(X_{i}(r))\cong k$, then the composition $\underline{\pi_{2}}\underline{\epsilon_{1}}\in\sthom_A(X_{i}(r),\tau^{-l} (X_{i}(r-l+j)))$ of the maps in the sequence \eqref{ses2} is non-zero in $\stmod A$ for $0<l<r$ and $j\geq 1$.
\end{Lem}
We remark that when $l=r$, the module $\tau^{-l}(X_{i}(r-l))$ is zero and so the composition $\underline{\pi_{2}}\underline{\epsilon_{1}}$ is zero as we can see from the sequence \eqref{ses1}.
\begin{Proof}
Applying $\sthom_A(X_i(r),-)$ to the triangle \eqref{ses2} yields a long exact sequence
\begin{align*}
\cdots \to \underline{( X_i(r), X_i(r))} &\xrightarrow{(\underline{\epsilon_1},\underline{\pi_1})^t_*} \underline{(X_i(r), X_i(r+j)\oplus X_{\overline{i+l}}(r-l))} \\ &\quad \xrightarrow{(\underline{\pi_2},\underline{\epsilon_2})_*} \underline{(X_i(r), X_{\overline{i+l}}(r-l+j))} \to \cdots.
\end{align*}
Consider the map $\alpha:=(\underline{\epsilon_1},0)^t: X_i(r)\to X_i(r+j)\oplus X_{\overline{i+l}}(r-l)$.
Then we have $(\underline{\pi_2},\underline{\epsilon_2})_*\big(\alpha\big) = (\underline{\pi_2},\underline{\epsilon_2})(\underline{\epsilon_1},0)^t = \underline{\pi_2\epsilon_1}$. Suppose on the contrary that $\underline{\pi_2\epsilon_1}=0$. It then follows from the exactness of the above long exact sequence that there is some $\gamma\in \stend_A(X_i(r))$ such that $(\underline{\epsilon_1},\underline{\pi_1})^t\gamma = (\underline{\epsilon_1},0)^t$.
Since $\stend_A(X_{i}(r))\cong k$, $\gamma$ is an isomorphism.
This means that $\underline{\pi_{1}}=0$ in $\stmod{A}$, which contradicts Lemma \ref{sectional-path-mor}.
\end{Proof}

\section{Proving Theorem \ref{main-result-1}}

We start by showing two easy implications of the results from the previous sections.
These will give strong restrictions on the modules of $\mathcal{C}$ that lie in an sms $\mathcal{S}$.
This in turn will give us some ideas on what sequence of modules we should consider to apply Proposition \ref{main-strat}.
Then the remaining of the section will be devoted to showing the candidate sequence satisfies the conditions of Proposition \ref{main-strat}.

\begin{Lem}\label{r<n}
$X_i(r)$ is not a stable brick for all $r>n$.
In particular, such a module cannot be a member of an sms of $\stmod{A}$.
\end{Lem}
\begin{Proof}
The first statement is a direct consequence of \cite[Lemma 3.5.1]{EK}, but we can also present the proof easily as follows. Consider $r>n$ and assume $\stend_A(X_i(r))\cong k$.
By taking $(l,j)=(n,n)$ in Lemma \ref{stbrick-induces-hom}, implies that map $\underline{\pi_2\epsilon_1}$ (with notation as in Lemma \ref{stbrick-induces-hom}) is a non-zero endomorphism of $X_i(r)$ in $\stmod{A}$.
By Lemma \ref{sectional-path-mor}, $\underline{\pi_2}, \underline{\epsilon_1}$ are non-zero morphisms in the radical of $\stmod{A}$, $\underline{\pi_2\epsilon_1}$ cannot be the identity; this contradicts the assumption that $\stend_A(X_i(r))\cong k$.
\end{Proof}

\begin{Lem}\label{|S|<n}
If $X_{i}(r)$ belongs to some stable semibrick $\mathcal{S}$ of $\stmod{A}$ for some $i\in\{1,\ldots, n\}$ and $r>1$, then the cardinality of $\mathcal{S}\cap\mathcal{C}$ is strictly less than $n$.
\end{Lem}
\begin{Proof}
Suppose we have $X_i(r)\in \mathcal{S}$ with $\mathcal{S}$ an sms of $\stmod{A}$ and $r>1$.
It follows from Lemma \ref{r<n} that we can assume $r\leq n$.

On one hand, taking $l=r-1$ in Lemma \ref{stbrick-induces-hom} we get (together with Lemma \ref{sectional-path-mor}) that $\sthom_A(X_{i}(r), X_{\overline{i+r-1}}(j+1))$ is non-zero for all $j\geq 0$.
Hence, by the orthogonality condition, none of the modules in the sectional path starting from $X_{\overline{i+r-1}}$ can be in $\mathcal{S}$.

On the other hand, by Lemma \ref{sectional-path-mor}, it follows from orthogonality condition that each sectional path starting at $X_k$ with $k\in\{1,\ldots, n\}$ can only have at most one module in $\mathcal{S}$.
In particular, $X_i(s)$ cannot be in $\mathcal{S}$ for all $s\neq r$.

The assumption of $r\leq n$ implies that $\overline{i+r-1}\neq i$, so the previous paragraphs amount to say that modules in $\mathcal{C}\cap (\mathcal{S}\setminus\{X_i(r)\})$ must come from the $n-2$ sectional paths starting at $X_j$ for $j\neq i, \overline{i+r-1}$, each of which has at most one module in $\mathcal{S}$; that is, $|\mathcal{C}\cap \mathcal{S}|\leq n-2+1=n-1$, as required.
\end{Proof}

It now follows from Lemma \ref{|S|<n} that we only need to consider
\[
\text{a stable semibrick $\mathcal{S}$ that contains all the quasi-simples }
\{ X_1, X_2, \ldots, X_n \}\text{ of }\mathcal{C}.
\]

\begin{Ex}
Note that for any $n\geq 1$, there are abundance of examples where the set of quasi-simples form a stable semibrick.
For example, consider the algebra used in Remark \ref{rem:bdd} with $n> 1$, then there is such a component whose quasi-simples are given by the simple modules $S_i$ with $i\in\{1, 2, \ldots, n-1\}$ and a length 2 module $M$ whose top is $S_{n+1}$ and socle is $S_n$.
This example works in the case of $n=1$, too.
Namely, we have $A$ the trivial extension algebra of the Kronecker algebra $K:=k(\xymatrix@1{2\ar@<3pt>[r] \ar@<-3pt>[r]& 1})$, and the length 2 module $M$ is the regular $K$-module (regarded as an $A$-module) with parameter $(1:0)\in\mathbb{P}^1_k$ - this is clearly a stable brick lying at the end of a homogeneous tube.
\end{Ex}

In view of the conditions of Proposition \ref{main-strat}, we want to find a sequence of indecomposable non-projective modules such that each one of them have a non-zero morphism in $\stmod{A}$ to one of the quasi-simples.
The following lemma gives us one possible candidate.

\begin{Lem}  \label{Lem-Omega-triangle}
For every positive integer $r\geq 1$ and $i\in \{1,\ldots,n\}$, there is a non-split triangle in $\stmod{A}$:
\begin{equation}\label{omega-triangle}
\Omega(X_{i}(r+1)) \to \Omega(X_{\overline{i+1}}(r))\to X_{i}\to.
\end{equation}
\end{Lem}
\begin{Proof}
Taking $(l,j)=(r,1)$ in the non-split triangle \eqref{ses1} yields
\[
X_{i}\xrightarrow{\underline{\epsilon}} X_{i}(r+1) \xrightarrow{\underline{\pi}} X_{\overline{i+1}}(r)\to\Omega^{-1}(X_{i}),
\]
which can then be rotated to form a non-split triangle with terms agreeing those in the triangle \eqref{omega-triangle} of the claim.  So it remains to show that the connecting morphism $\Omega(X_{\overline{i+1}}(r))\to X_{i}$ of the above triangle is non-zero.
Suppose the contrary, then the cone of the zero connecting morphism will be $X_{\overline{i+1}}(r)\oplus X_{i}\ncong X_{i}(r+1)$, hence a contradiction.
\end{Proof}

This means that our candidate sequence should be
\begin{align}\label{main-seq1}
(M_r)_{r\geq 1} \;\;\text{ with }\;\; M_r:= \Omega( X_{\overline{i-r}}(r+1)),
\end{align}
for any fixed $i\in\{1, 2, \ldots, n\}$.

\begin{Lem}\label{condi}
The sequence \eqref{main-seq1} satisfies condition (i) of Proposition \ref{main-strat}.
\end{Lem}
\begin{Proof}
From the triangle \eqref{omega-triangle} of Lemma \ref{Lem-Omega-triangle}, we see that $\sthom_A(\Omega(X_{\overline{i-r}}(r)), X_{\overline{i-r-1}})\neq 0$, so $\Omega(X_{\overline{i-r}}(r))$ can not be isomorphic to any $S\in \mathcal{S}\setminus\{ X_1, X_2, \ldots, X_n\}$.

Moreover, by the difference in quasi-lengths, this sequence is clearly pairwise non-isomorphic, and none of them is isomorphic to the quasi-simples $X_1, X_2, \ldots, X_n$ of $\mathcal{C}$.
\end{Proof}

The remaining of this subsection is devoted to show that the sequence \eqref{main-seq1} satisfies the condition (ii) in Proposition \ref{main-strat}, namely, to show that the triangle obtained in Lemma \ref{omega-triangle} are the only ones we need.
For this purpose, we first observe the following behaviour.

\begin{Lem}  \label{orthogonality}
Consider the set $\{X_{i}\mid i=1, 2,\ldots, n\}$ of quasi-simples of $\mathcal{C}$.
If $S$ is a non-projective indecomposable module, then we have the following.
 \begin{enumerate}[$(a)$]
\item If $\sthom_A(X_{i},S)=0$ for all $i\in\{1,\ldots,n\}$, then $\sthom_A(X_{i}(r),S)=\sthom_A([r]X_{i},S)=0$ for all $r\geq 1, i\in\{1,\ldots,n\}$.
\item If $\sthom_A(S,X_{i})=0$ for all $i\in\{1,\ldots,n\}$, then $\sthom_A(S,X_{i}(r))=\sthom_A(S,[r]X_{i})=0$ for all $r\geq 1, i\in\{1,\ldots,n\}$.
 \end{enumerate}
\end{Lem}

\begin{Proof}
We only prove (a), since (b) is dual. We prove it by induction on $r$. Fix some $i$ in $\{1,\ldots,n\}$.
For $r=1$, $X_{i}(r)=X_{i}$ and the statement follows by the assumption that $\sthom_A(X_{i},S)= 0$.
Taking $(l,j)=(r-1,1)$ in the sequence \eqref{ses1} yields
\begin{equation}\label{ses3}
 X_{i}(r-1) \xrightarrow{} X_{i}(r) \xrightarrow{} \tau^{-(r-1)} X_{i}  \xrightarrow{}  \Omega^{-1}( X_{i}(r-1)),
\end{equation}
and we obtain a long exact sequence by applying $\sthom_A(-,S)$ to it.
By induction hypothesis we have $\sthom_A(X_i(r-1), S)=0$.
Since $\tau^{-(r-1)}X_{i}= X_{\overline{i+r-1}}$ by definition, it follows from the assumption that $\sthom_A(\tau^{-(r-1)}X_i, S)=0$.
Hence, $\sthom_A(X_i(r), S)=0$ follows from the exactness of the long exact sequence.  The proof for the statement on $\sthom_A([r]X_i, S)$ is similar.
\end{Proof}

Let $\nu=D\mathrm{Hom}_A(-,A)$ denote the Nakayama functor.  Then we have for all $M,N\in\stmod A$ the bifunctorial isomorphisms:
\begin{equation}\label{serre-duality}
\sthom_A(M,N) \cong D\mathrm{Ext}_A^1(N,\tau M) \cong D\sthom_A(N,\nu\Omega M),
\end{equation}
where the first isomorphism is just the Auslander-Reiten duality and the second isomorphism follows from the well-known fact that $\tau\cong\nu\Omega^2$.
Note that this shows that we have Serre duality in the triangulated category $\stmod{A}$, where the Serre functor is $\nu\Omega$.

\begin{Lem}  \label{Omega-orthogonality}
If $S$ is a non-projective indecomposable module with $\sthom_A(S, X_{i})=0$ for all $i\in\{1,\ldots,n\}$, then $\sthom_A(\Omega(X_{i}(r)),S)= 0$ for all $r\geq 1$ and for all $i\in\{1,\ldots,n\}$.
In particular, any triangle of the form $N\to M_r\to S\to $ splits for all $S\in \mathcal{S}\setminus\{X_1, X_2, \ldots, X_n\}$.
\end{Lem}
\begin{Proof}
Take $M,N$ as $\Omega(X_{i}(r)), S$ respectively in (\ref{serre-duality}), then we get that
\[
\dim\sthom_A(\Omega(X_{i}(r)),S) = \dim\sthom_A(S,\nu\Omega(\Omega X_{i}(r))).
\]
Since $\nu\Omega^2 X_{i}(r)\cong \tau X_{i}(r)\cong X_{\overline{i-1}}(r)$, the later space is just $\sthom_A(S,X_{\overline{i-1}}(r))$, which is zero by Lemma \ref{orthogonality}.
\end{Proof}

Now, it remains to show that the triangles \eqref{omega-triangle} are the only ones induced by a morphism in $\stmod{A}$ from $M_r$ to the quasi-simples of $\mathcal{C}$ (hence, to the elements of $\mathcal{S}$).

\begin{Lem}  \label{dimensionformula2}
Suppose that the quasi-simples $\{X_i\mid i=1,\ldots, n\}$ in $\mathcal{C}$ is a stable semibrick.
For all $i,j\in\{1,\ldots,n\}$ and for all $r\geq 1$, we have the following.

\begin{enumerate}[(i)]
\item $\sthom_A(X_{j},X_{i}(r))\cong k\delta_{i,j}$ and $\sthom_A([r]X_{i},X_{j})\cong k\delta_{i,j}$;

\item $\sthom_A(\Omega(X_{\overline{i+1}}(r)),X_{j})\cong k\delta_{i,j}$,
\end{enumerate}
where $k\delta_{i,j}=k$ if $i=j$, zero otherwise.
\end{Lem}

\begin{Proof} (i) We only prove $\sthom_A(X_{j},X_{i}(r))\cong k\delta_{i,j}$, the other one can be proved dually.

Fix an $i\in\{1,\ldots,n\}$.
For ease of exposition, denote by $\epsilon_s:X(s-1)\to X(s)$ the irreducible map that induces the triangle \eqref{ses3}, and
$\theta_s := \epsilon_{s}\epsilon_{s-1}\cdots \epsilon_1$ the composition of them.

We now proof the claim by induction on $r$. For $r=1$, this is trivial.  Suppose that $r > 1$.
Applying $\sthom(X_{j},-)$ to the triangle (\ref{ses3}), we obtain an exact sequence
{\small\[
\underline{(X_{j},X_{i}(r-1))}\xrightarrow{\,\underline{\epsilon_{r*}}\,} \underline{(X_{j},X_{i}(r))}\xrightarrow{\,\underline{\pi_{r*}}\,} \underline{(X_{j},X_{\overline{i+r-1}})}\xrightarrow{\,\underline{\rho_{r *}}\,}\underline{(X_{j},\Omega^{-1}(X_{i}(r-1)))}.
\]}
By the induction hypothesis, we have $\sthom_A(X_j,X_i(r-1))\cong k\delta_{i,j}$.

Let us consider the case when $i=j$.
Recall from Lemma \ref{sectional-path-mor} and Remark \ref{dually} that $\underline{\theta_s}$ is always non-zero.
Therefore, we have a non-zero map \[
\underline{\epsilon_r}_*(\underline{\theta_{r-1}}) = \underline{\epsilon_r\theta_{r-1}} = \underline{\theta_r}.\]
So $\sthom_A(X_j,X_i(r-1))$ being one-dimensional implies that ${\underline{\epsilon_r}}_*$ is injective.

Regardless of whether $i=j$, we consider the relation between $j$ and $\overline{i+r-1}$.

(a) If $j \neq \overline{i+r-1}$, then $\sthom_{A}(X_{j},X_{\overline{i+r-1}})=0$ by the assumption on the orthogonality of the quasi-simples.
Hence, it follows that $\sthom_A(X_{j}, X_{i}(r))\cong k\delta_{i,j}$.

(b) If $j=\overline{i+r-1}$, then $\sthom_{A}(X_{j},X_{\overline{i+r-1}})$ is one-dimensional and spanned by the identity map $\mathrm{id}$.
Since the triangle \eqref{ses1} is non-split, the morphism $\underline{\rho_{r}}$ is non-zero. By the following simple calclation,
\[
{\underline{\rho_{r}}}_*(\mathrm{id})=\underline{\rho_{r}}\mathrm{id}= \underline{\rho_r},\]
we see that ${\underline{\rho_r}}_*$ is non-zero, and so one-dimenionality of $\sthom_{A}(X_{i},X_{\overline{i+r-1}})$ implies that it is injective.
It now follows that $\sthom_A(X_{j},X_{i}(r))\cong \sthom_A(X_j,X_i(r-1))\cong k\delta_{i,j}$.

This completes the proof.

(ii)
Follows from (i) by applying Serre duality (\ref{serre-duality}):
\[
  \sthom_A(\Omega(X_{\overline{i+1}}(r)),X_{j})\cong
D\sthom_A(X_{j},\nu\Omega^2(X_{\overline{i+1}}(r))) \cong  D\sthom_A(X_{j},X_{i}(r)).
\]
\end{Proof}

\medskip

It should be clear now that Theorem \ref{main-result-1} follows.
For clarity, let us recap the argument of this proof.

\begin{Proof}[Proof of Theorem \ref{main-result-1}]
By Lemma \ref{|S|<n}, it suffices to show that if a stable semibrick $\mathcal{S}$ contains all the quasi-simples $\{X_1, X_2, \ldots, X_n\}$ of $\mathcal{C}$, then it cannot be an sms of $\stmod{A}$.
By showing that the sequence \eqref{main-seq1} satisfies conditions (i) and (ii) of Proposition \ref{main-strat}, such a claim is just an immediate consequence of Proposition \ref{main-strat}.

We have already explained in Lemma \ref{condi} why condition (i) holds.
For condition (ii), it follows from Lemma \ref{Omega-orthogonality} that there is no non-split triangle of the form $N\to M_r\to S\to $ with $S\in \mathcal{S}$ not isomorphic to a quasi-simple of $\mathcal{C}$.
So it remains to look at triangles of the form
\[
N\to \Omega(X_{\overline{i-r}}(r+1)) \to X_j,
\]
for $r\geq 1$.
By Lemma \ref{dimensionformula2} (ii), this triangle must split when $j\neq \overline{i-r-1}$; otherwise, it is isomorphic to the one given in \eqref{omega-triangle}. This shows that condition (ii) of Proposition \ref{main-strat} is satisfied; hence completing the proof of Theorem \ref{main-result-1}.
\end{Proof}

\bigskip

\section{Proving Theorem \ref{main-result-2}}

As before, we fix $A$ to be a ring-indecomposable non-simple self-injective algebra, $\mathcal{C}$ to be a stably quasi-serial component of rank $n\geq 1$ of the AR-quiver of $A$, and $\{X_{i} \mid i=1,\ldots, n\}$ the set of quasi-simples of $\mathcal{C}$ with $X_{\overline{i+1}}\cong \tau^{-1}(X_i)$.

Recall that $\Omega^{\pm 1}$ are stable auto-equivalences, so they map stably quasi-serial components to stably quasi-serial components and, in particular, they preserve the quasi-lengths of modules.
We will first show that the assumption of $X_i(n)$ being a stable brick implies that $\Omega^{\pm 1}$ cannot fix $\mathcal{C}$.
This helps to determine the dimensions of various stable Hom-spaces between certain modules in $\mathcal{C}$.
At the end, we will use these stable Hom-spaces to show that a certain sequence (see Assumption \ref{assump} (3)) of modules in $\Omega(\mathcal{C})$ satisfies the condition of Proposition \ref{main-strat}, and so proving Theorem \ref{main-result-2}.

Let us start by recalling the following terminology from \cite{EK}.
\begin{Def}
Let $j,l$ be positive integers.
The wing of $X_j(l)$ is the set of isoclasses of indecomposable non-projective modules in $\mathcal{C}$ given by
\begin{align*}
\mathcal{W}_{j,l} &:= \{ X_{\overline{j+d}}(h) \mid d\geq 0,\;\; h\geq 1,\;\; 1\leq d+h \leq l \} 
\end{align*}
\end{Def}

The following will be crucial to our calculation of the dimension of the stable Hom-spaces.

\begin{Lem}\label{dimsum}
Consider an indecomposable non-projective module $M$.
\begin{enumerate}[(i)]
\item {\rm \cite[2.2]{EK}} If $M,\Omega^{-1}(M)\notin \mathcal{W}_{\overline{i+1},r-1}$, then
\[
\dim\sthom_A( M, X_i(r)) = \sum_{j=0}^{r-1} \dim\sthom_A(M, X_{\overline{i+j}}).
\]

\item If $M,\Omega(M)\notin\mathcal{W}_{i,r-1}$, then
\[
\dim \sthom_A(X_i(r), M)= \sum_{j=0}^{r-1} \dim\sthom_A(X_{\overline{i+j}}, M).
\]
\end{enumerate}
\end{Lem}

While \cite{EK} did not state Lemma \ref{dimsum} (ii), its proof is completely dual to that of \cite[2.2]{EK}.

\begin{Lem}\label{Xl-Xj(n)}
The following are equivalent.
\begin{enumerate}[(i)]
\item $X_j(n)$ is a stable brick for some $j\in\{1, 2, \ldots, n \}$.
\item $X_j(n)$ is a stable brick for all $j\in\{1, 2, \ldots, n\}$.
\item $\sthom_A(X_l,X_j(n)) \cong k\delta_{j,l}$ for all $j,l\in\{1, 2, \ldots, n\}$.
\item[(iii')] $\sthom_A([n]X_l,X_j) \cong  k\delta_{j,l}$ for all $j,l\in\{1, 2, \ldots, n\}$.
\end{enumerate}
\end{Lem}
\begin{Proof}
\underline{(i) $\Leftrightarrow$ (ii)}:  The if direction is trivial.  The converse follows from the fact that $\tau$ is an auto-equivalence of $\stmod{A}$.

\underline{(ii) $\Leftrightarrow$ (iii)}:
The quasi-length of $M=X_j(n)$ and $\Omega(M)$ is clearly larger than any member of $\mathcal{W}_{j,n-1}$, so we can apply Lemma \ref{dimsum} (ii) and get that
\[
\dim\sthom_A( X_j(n),X_j(n)) = \sum_{h=0}^{n-1} \dim\sthom_A(X_{\overline{j+h}},X_{j}(n)).
\]
By Remark \ref{dually}, we have $\sthom_A(X_j,X_j(n))\neq 0$, so the claimed equivalence follows from the displayed equality.

\underline{(ii) $\Leftrightarrow$ (iii')}:
Similar to the previous one (but use Lemma \ref{dimsum} (i) and Lemma \ref{sectional-path-mor} instead).
\end{Proof}

This allows us to exclude the situation when (co)syzygy sends non-projective indecomposable modules in $\mathcal{C}$ to $\mathcal{C}$.

\begin{Lem}\label{omega-fix-C}
If $X_i(n)$ is a stable brick, then $\Omega^{-1}(X_i)\ncong X_j$ for any $j$, i.e. $\Omega^{-1}(\mathcal{C})\neq \mathcal{C}$.
\end{Lem}
\begin{Proof}
Suppose on the contrary that $\Omega^{-1}(X_i)\cong X_{\overline{i+d}}$ for some $1\leq d\leq n$.
Since $\Omega^\pm$ are auto-equivalences on $\stmod{A}$ that commute with the Auslander-Reiten translation, the assumption of $\Omega^{-1}(X_i)\cong X_{\overline{i+d}}$ means that we have $\Omega^{-1}(X_{\overline{j-d}}(r))\cong X_j(r)$ for all $j,r$.
In particular, we have the following isomorphisms.
\begin{align*}
\sthom_A(X_i,X_j(n)) & \cong \sthom_A(X_i, \Omega^{-1}(X_{\overline{j-d}}(n))) \cong \sthom_A(\Omega(X_i), X_{\overline{j-d}}(n)) \\
& \cong D\sthom_A(X_{\overline{j-d}}(n), X_{\overline{i-1}}) \cong D\sthom_A([n]X_{\overline{j-d-1}}, X_{\overline{i-1}}),
\end{align*}
where the third isomorphism comes from Serre duality and the last isomorphism follows from the fact that $X_a(b)\cong [b]X_{\overline{a+b-1}}$.

By Lemma \ref{Xl-Xj(n)} (iii), the first space in the formula above is given by $k\delta_{i,j}$, whereas Lemma \ref{Xl-Xj(n)} (iii') says that the last space in the formula above is $k\delta_{j,\overline{i+d}}$.
Hence, we must have $i=\overline{i+d}$, meaning that $\Omega(X_i(n))\cong X_i(n)$.

Consider the non-split triangle \eqref{ses1} with $(l,j)=(n,mn)$.
Since every term the triangle is indecomposable, non-splitness implies that the connecting morphism $\Omega(X_i(n))\to X_i(n)$ is non-zero.
Moreover, it cannot be an isomorphism; otherwise, we have $X_i(2n)\cong 0$, which is absurd.
Hence, the $\underline{\mathrm{End}}_A(X_i(n))\cong \sthom_A(\Omega(X_i(n)), X_i(n))$ is at least two-dimensional; a contradiction.
\end{Proof}

\begin{Assumption}\label{assump}
From now on, we will assume the following unless otherwise stated.
\begin{enumerate}[(1)]
\item $\mathcal{S}$ is a stable semibrick containing $X_i(n)$ for some fixed $i\in\{1,2, \ldots, n\}$.  In particular, we have $\Omega(\mathcal{C})\neq \mathcal{C}$ (by Lemma \ref{omega-fix-C}).

\item There is a sequence of integers $n=j_0 > j_1 >\cdots > j_a \geq 1$  such that $\mathcal{S}$ contains
\[
S_t:=\Omega(X_{\overline{i+j_t}}(j_{t-1}-j_t))
\]
for all $1\leq t \leq a$, and $a\geq 0$ is the maximum non-negative integer such that this property holds, i.e.  $\Omega(X_{\overline{i+j}}(j_a-j))\notin \mathcal{S}$ for all $1\leq j <j_a$.

\item For a non-negative integer $l\in\mathbb{Z}_{\geq 0}$, write $l=(a+1)m+t$ with $m \in\mathbb{Z}_{\geq 0}$ and $0\leq t\leq a$, define
\[
M_l := \Omega( X_i(mn+j_t) ).
\]
\end{enumerate}
\end{Assumption}

Note that the sequence $j_0 > j_1 > \cdots j_a >0$ is uniquely defined.
Indeed, since Lemma \ref{sectional-path-mor} along with Remark \ref{dually} says that each infinite sectional path contains only at most one module, $S_1$ is given by the unique module in the infinite sectional path ending at $\Omega(X_{\overline{i-1}})$ and in the wing of $\Omega(X_i(n))$.
Thus, $j_1$ is well-defined and unique if such an $S_1$ exist.
Moreover, inductively, the existence of $S_t$ for $1\leq t<a$ uniquely determines $S_{t+1} = [j_{t}-j_{t+1}]\Omega(X_{\overline{i+j_{t}-1}})$, if $t+1\leq a$, since $S_{t+1}$ is the only member of $\mathcal{S}$ lying on the infinite sectional path ending at $\Omega(X_{\overline{i+j_{t}-1}})$ (and in the wing of $\Omega(X_i(n))$).  Thus, $j_{t+1}$ is uniquely determined by $j_t$.
The uniqueness of the sequence now follows from the uniqueness of each $j_t$ and the maximality of $a$.

Of course we have deliberately picked the notation so that our goal is to show that the sequence $(M_l)_{l\geq 0}$ satisfies the conditions of Proposition \ref{main-strat}.
From the experience of proving Theorem \ref{main-result-1}, one probably could guess that the sequence  $(\Omega(X_i(rn)))_{r\geq 1}$ may be suitable to prove Theorem \ref{main-result-2}.
Although this is true in certain cases (namely, the case when $a=0$ in the notation of Assumption \ref{assump}), this sequence will not be sufficient in general, as there could be other non-split triangles induced by non-zero morphisms from $\Omega(X_i(rn))$ to other modules in $\Omega(\mathcal{C})$.
Indeed, the following lemma, which lists all the triangles needed for Proposition \ref{main-strat}, gives an explanation of why we consider the sequence $(M_l)_{l\geq 0}$ defined above instead.

\begin{Lem}\label{sms-triangles}
We have the following two non-split triangles for any integer $m\geq 0$.
\[\arraycolsep=1.4pt
\begin{array}{rccccccl}
\text{(a) \;\;} & \Omega(X_{i}((m+1)n+j_{t})) & \to & \Omega(X_{i}(mn+j_{t})) & \to & X_{i}(n) & \to & \text{for all }0\leq t\leq a.\\ [1ex]
\text{(b) \;\;} & \Omega(X_{i}(mn+j_{t+1}))& \to & \Omega(X_{i}(mn+j_{t})) & \to & S_{t+1}& \to & \text{for all }0\leq t< a.
\end{array}
\]
\end{Lem}
\begin{Proof}Same argument as in Lemma \ref{Lem-Omega-triangle} works.
Here, the triangle (a) is obtained by rotating the non-split triangle \eqref{ses1} with $(l,j)=(n, mn+j_{t})$; whereas the triangle (b) is obtained by rotating the non-split triangle \eqref{ses1} with $(l,j)=(mn+j_{t+1}, j_{t}-j_{t+1})$.
\end{Proof}

If any non-split triangle $N\to M_l\to S\to $ with $S\in \mathcal{S}$ and $l\geq 0$ is isomorphic to one of the forms in Lemma \ref{sms-triangles}, then condition (ii) of Proposition \ref{main-strat} is also satisfied.
For this purpose, we need to show that $\dim \sthom_A(M_l, S)$ for $S\in \mathcal{S}$ must be either zero or one, with one appearing precisely as described in Lemma \ref{sms-triangles}.
We will devote the remaining of this section (up to the proof of Theorem \ref{main-result-2}) to prove this.

We start with a strengthened version of Lemma \ref{Xl-Xj(n)}.

\begin{Lem}\label{Xi(n)-stbrick}
The following are equivalent.
\begin{enumerate}[(i)]
\item $X_j(n)$ is a stable brick for some $j\in\{1, 2, \ldots, n \}$.
\item $X_j(n)$ is a stable brick for all $j\in\{1, 2, \ldots, n\}$.
\item $\sthom_A(X_l,X_j(n)) \cong k\delta_{j,l}$ for all $j, l\in\{1, 2, \ldots, n\}$.
\item[(iii')] $\sthom_A([n]X_l,X_j) \cong k\delta_{j,l}$ for all $j, l\in\{1, 2, \ldots, n\}$.
\item $\sthom_A(X_l,X_j(r)) \cong k\delta_{j,l}$ for all $j, l\in\{1, 2, \ldots, n\}$ and $r\geq 1$.
\item[(iv')] $\sthom_A([r]X_l,X_j) \cong k\delta_{j,l}$ for all $j, l\in\{1, 2, \ldots, n\}$ and $r\geq 1$.
\item $\{X_j\mid j=1, 2, \ldots,n \}$ is a stable semibrick.
\end{enumerate}
\end{Lem}
\begin{Proof}
The equivalences between (i), (ii), (iii), and (iii') are already shown in Lemma \ref{Xl-Xj(n)}.

The remaining equivalences will be shown in the following way:
\begin{center}
 (iv) $\Rightarrow$ (iii) $\Rightarrow$ (v) $\Rightarrow$ (iv)
\end{center}
and an analogous one where (iii') and (iv') are replaced by (iii') and (iv'); we will omit the arguments in this analogous setting as they are similar to the one above.

\underline{(iv) $\Rightarrow$ (iii)}: Trivial.

\underline{(iii) $\Rightarrow$ (v)}:
Take any $j\in\{1, 2, \ldots, n\}$ and $M=X_j$.
Any quasi-simple in $\mathcal{W}_{\overline{j+1},n-1}$ (resp. $\mathcal{W}_{j,n-1}$) is of the form $X_{\overline{j+h}}$ for some $h\in\{1, 2, \ldots, n-1\}$, so it cannot be isomorphic to $M$.
Since (iii) is equivalent to (i), it follows from Lemma \ref{omega-fix-C} that $\Omega^{-1}(M)$ lies in a component distinct from $\mathcal{C}$, so it cannot be in $\mathcal{W}_{\overline{j+1},n}$.
Now we can apply Lemma \ref{dimsum} (i) and get that
\[
1=\dim\sthom_A( X_{j},X_{j}(n)) = \sum_{h=0}^{n-1} \dim\sthom_A(X_j,X_{\overline{j+h}}).
\]
Since $\underline{\mathrm{End}}_A(X_j(n))\neq 0$, it follows from the above equation that $\sthom_A(X_j, X_l)\cong k\delta_{j,l}$.

\underline{(v) $\Rightarrow$ (iv)}: This is just Lemma \ref{dimensionformula2} (i).
\end{Proof}

We can now show that any triangle $N\to M_l\to S\to $ splits for any $S\ncong X_i(n)$ in $\mathcal{S}$ but not in $\Omega(\mathcal{W}_{\overline{i+1},n-1})$.

\begin{Lem}  \label{OmegaXi(mn)}
Suppose $S$ is a non-projective indecomposable module such that
\[\sthom_A(X_{i}(n),S)=0=\sthom_A(S,X_{i}(n)) \quad \text{and}\quad S\notin \Omega(\mathcal{W}_{\overline{i+1},n-1}).\]
Then $\sthom_A(\Omega(X_j(r)), S)=0$ for all $r\geq 1$ and any $j\in\{1,\ldots, n\}$.
\end{Lem}
\begin{Proof}
First, we consider the case when $S$ is not in $\mathcal{W}_{\overline{i+1},n-1}$.
Combining with the assumption that $\Omega^{-1}(S)\notin \mathcal{W}_{\overline{i+1}, n-1}$, it follows from Lemma \ref{dimsum} (i) that $\sthom_A(S, X_{j})=0$ for any quasi-simple $X_j$.
So, by Lemma \ref{orthogonality}, we have $\sthom_A(S, X_{j}(r))=0$ for any $X_j(r)\in\mathcal{C}$.
Hence, using Serre duality we can see that $\sthom_A(\Omega(X_j(r)), S)=0$ for any $j\in\{1,\ldots, n\}$ and $r\geq 1$.

\medskip

Next, we consider the case when $S=X_{\overline{i+j}}(s)\in\mathcal{W}_{\overline{i+1},n-1}$.

If $j+s=n$ (i.e. $S\notin \mathcal{W}_{\overline{i+1},n-2}$), then $S=X_{\overline{i+j}}(s)=[s]X_{\overline{i-1}}$ lies in the sectional path that contains $X_i(n)=[n]X_{\overline{i-1}}$ and ends at $X_{\overline{i-1}}$.
For such $S$, it follows from Lemma \ref{sectional-path-mor} that $\sthom_A(X_i(n),S)\neq 0$, which contradicts the assumption on $S$.

We can now assume that $S\in\mathcal{W}_{\overline{i+1},n-2}$.
For any $r\geq 1$, Serre duality says that \[
\sthom_A(\Omega(X_i(r)), X_{\overline{i+j}}(s))\cong D\sthom_A( X_{\overline{i+j}}(s), X_{\overline{i-1}}(r)),
\]
so it suffices to show that $\sthom_A( X_{\overline{i+j}}(s), X_{\overline{i-1}}(r))=0$.

Taking $M=X_{\overline{i-1}}(r)$, then we have $M\notin \mathcal{W}_{\overline{i+1},n-2}$.  We also have $\Omega(M)\notin \mathcal{C}$ by Assumption \ref{assump} (1) and Lemma \ref{omega-fix-C}.
Now we can apply Lemma \ref{dimsum} (ii) and get that
\[\dim\sthom_A( X_{\overline{i+j}}(s),X_{\overline{i-1}}(r))=\sum_{h=0}^{s-1} \dim\sthom_A(X_{\overline{i+j+h}},X_{\overline{i-1}}(r)).\]
Since $X_i(n)$ is a stable brick, it follows from Lemma \ref{Xi(n)-stbrick} (v) that the right-hand side is non-zero if and only if there is some $h\in\{0, 1, \ldots, s-1\}$ with $\overline{i+j+h}=\overline{i-1}$.
But the condition $j+s<n$ implies that $j+h<n-1$, and so this is impossible.
\end{Proof}

Now we consider the stable Hom-spaces from $M_l$'s to modules in $\Omega(\mathcal{W}_{\overline{i+1},n-1})$.

\begin{Lem}\label{hi-hom}
The following holds for any $r\geq 1$ and $X_{\overline{i+j}}(s)\in\mathcal{W}_{\overline{i+1},n-1}$:
 \begin{align*}
\sthom_A(X_i(r), X_{\overline{i+j}}(s)) \cong& \begin{cases}
k, & \text{if }j\leq \overline{r-1} < j+s \text{ (or equivalently, $X_{\overline{i+r-1}}\in \mathcal{W}_{\overline{i+j},s}$)};\\
0, & \text{otherwise.}
\end{cases}
\end{align*}

\end{Lem}
\begin{Proof}
Taking $M=X_i(r)$ means that $M\notin \mathcal{W}_{\overline{i+j+1},s-1}$.
Assumption \ref{assump} (1) and Lemma \ref{omega-fix-C} implies that $\Omega^{-1}(M)\notin \mathcal{W}_{\overline{i+j+1},s-1}$.
So we can apply Lemma \ref{dimsum} (i) and get that
\begin{align}
\dim \sthom_A(X_i(r), X_{\overline{i+j}}(s)) &= \sum_{h=0}^{s-1} \dim \sthom_A(X_i(r), X_{\overline{i+j+h}}) \notag \\
&= \sum_{h=0}^{s-1} \dim \sthom_A([r]X_{\overline{i+r-1}}, X_{\overline{i+j+h}}), \label{swap-not}
\end{align}
where the second equality follows from the fact that $X_a(b)\cong [b]X_{\overline{a+b-1}}$ for all $a,b$.

Since $j+s<n$, if there is $h\in\{0, 1, \ldots, s-1\}$ such that $\overline{j+h}=\overline{r-1}$ (equivalently, $\overline{r-1}\in \{j, j+1, \ldots, j+s-1\}$), then it is unique.
As $X_i(n)$ is a stable brick, by using Lemma \ref{Xi(n)-stbrick} (v'), we can see that every term in \eqref{swap-not} is 0 when there is no $h\in\{0, 1, \ldots, s-1\}$ with $j+h=\overline{r-1}$; otherwise, all but one term is 0 with the remaining one being $1$.
This completes the proof.
\end{Proof}

Note that since $\tau$ is an auto-equivalence on $\stmod{A}$, we can freely shift both modules simultaneously in the formula of Lemma \ref{hi-hom} in the ``$\tau$-direction", i.e. replace $i$ by any other $i'\in\{1, 2, \ldots, n\}$ everywhere in the statement.

\begin{Lem}\label{S_t-hom}
The following holds for any $1\leq t\leq a$ and $X_{\overline{i+j}}(s)\in\mathcal{W}_{\overline{i+1},n-1}$:
\[
\sthom_A(X_{\overline{i+j}}(s), X_{\overline{i+j_t}}(j_{t-1}-j_t)) \cong
\begin{cases}
k\delta_{t,b}, & \text{if $j_b < s+j \leq j_{b-1}$ for some $1\leq b\leq a$;}\\
0, & \text{otherwise.}
\end{cases}
\]
\end{Lem}
\begin{Proof}
Take $i':=\overline{i+j}$, then $X_{\overline{i+j_t}}(j_{t-1}-j_t)\in \mathcal{W}_{\overline{i'+1},n-1}$, so it follows from Lemma \ref{hi-hom} says that $\sthom_A(X_{i'}(s), X_{\overline{i'+j_t-j}}(j_{t-1}-j_t))$ is non-zero (in which case, is of dimension $1$) if and only if \[
j_t-j\leq s-1 < j_{t-1}-j_t+j_t-j=j_{t-1}-j,
\]
holds.
This condition is equivalent to $j_t+1\leq s+j < j_{t-1}+1$, so the claim follows.
\end{Proof}

\begin{Lem}\label{OmegaConfig}
If $S\in \mathcal{S}\cap\Omega(\mathcal{C})$, then $S$ satisfies one of the following (mutually exclusive) conditions.
\begin{enumerate}[(i)]
\item $S\cong S_t$ for some $1\leq t\leq a$.
\item $S\in \Omega(\mathcal{W}_{\overline{i+j_t+1},j_{t-1}-j_t-2})$ for some $1\leq t\leq a$.
\item $S\in \Omega(\mathcal{W}_{\overline{i+1}, j_a-2})$.
\end{enumerate}
\end{Lem}
\begin{Proof}
For arbitrary module $X_{\overline{i+l}}(r)\in\mathcal{C}$, it follows from Serre duality and Lemma \ref{stbrick-induces-hom} that the space \[\sthom_A(\Omega(X_{\overline{i+l}}(r)), X_{i}(n))\cong D\sthom_A(X_{i}(n), \tau (X_{\overline{i+l}}(r)))\]
is non-zero if $\tau(X_{\overline{i+l}}(r))\notin \mathcal{W}_{i,n-1}$.
Thus, $S$ being in the stable semibrick $\mathcal{S}$ implies that $S\in \Omega(\mathcal{W}_{\overline{i+1},n-1})$.

We claim that such an $S=\Omega(X_{\overline{i+j}}(s))$ satisfies only one of the three possibilities:
\begin{enumerate}[(a)]
\item $\sthom_A(S,S_t)\cong k$ for some $1\leq t\leq a$;
\item $j+s = j_a$;
\item $S$ lies in one of the $a+1$ wings described in (ii) and (iii);
\end{enumerate}
Hence, $\mathcal{S}$ being a stable semibrick and $S$ satisfying (a) implies that $S\cong S_t$, whereas $S$ satisfying (b) cannot be in $\mathcal{S}$ by Assumption \ref{assump} (2), and so the claim of the lemma follows.

Indeed, if $1\leq j\leq j_a$, then $S\notin \Omega(\mathcal{W}_{\overline{i+1}, j_a-2})$ (hence, does not satisfy (c)) implies that it either satisfies (b) or does not lies in $ \Omega(\mathcal{W}_{\overline{i+1},j_a-1})$.
In the later case, we then have $j+s>j_a$, and so there must be some $0\leq t\leq a$ with $j_t<j+s\leq j_{t-1}$.
Now it follows from Lemma \ref{S_t-hom} that $\sthom_A(S,S_{t-1})\cong k$, i.e $S$ satisfies (a).

Similarly, if $S\notin \Omega(\mathcal{W}_{\overline{i+j_t+1},j_{t-1}-j_{t}-2})$ and $j_t<j<j_{t-1}$ for some $1\leq t\leq a$, then Lemma \ref{S_t-hom} implies that $\sthom_A(S,S_{t'})\cong k$ for some $0\leq t'\leq t$.
\end{Proof}

\begin{Lem}\label{OmegaHom}
We have the following isomorphisms of stable Hom-spaces.
\begin{enumerate}[(i)]
\item For any $0\leq t\leq a$ and $m\geq 0$, we have
\[
\sthom_A(\Omega(X_i(mn+j_t)), X_i(n))\cong k.
\]

\item If $S\in \mathcal{S}\setminus\{X_i(n)\}$, then for any $0\leq t\leq a$ and $m\geq 0$, we have
\[
\sthom_A(\Omega(X_i(mn+j_t)), S) \cong \begin{cases}
k, &\text{if }S\cong S_{t+1}\text{ and }t<a;\\
0, & \text{otherwise.}
\end{cases}
\]
\end{enumerate}
\end{Lem}
\begin{Proof}
(i) By Serre duality, it suffices to show that $\dim \sthom_A(X_i(n), X_{\overline{i-1}}(mn+j_t))=1$.
For $m=0$, since $n-1\leq \overline{n-1} < n-1+j_t$ always hold, we get the required dimension by Lemma \ref{hi-hom}.
For $m>0$, we can apply Lemma \ref{dimsum} to get that
\[
\dim \sthom_A(X_i(n), X_{\overline{i-1}}(mn+j_t)) = \sum_{h=0}^{n-1} \dim \sthom_A(X_{\overline{i+h}}, X_{\overline{i-1}}(mn+j_t)).
\]
Since $\overline{i+h}$ runs through all $1, 2, \ldots, n$ exactly once, it follows from Lemma \ref{Xi(n)-stbrick} (v) that there is only one non-zero number, which is 1, in the summation.

\medskip

(ii)  Let us consider first the case when $S\notin \mathcal{S}\cap \Omega(\mathcal{C})$.
Then $S$ satisfies the conditions of Lemma \ref{OmegaXi(mn)}, and so (taking $r=mn+j_t$) we get the vanishing of the stable Hom-space as claimed.

Up to the end of the proof, we assume that $S\in \mathcal{S}\cap \Omega(\mathcal{C})$.
Note that this means that $S\ncong X_i(n)$ by Lemma \ref{omega-fix-C}.
By Lemma \ref{OmegaConfig}, we only need to consider the case when (a) $S\cong S_r$ for some $1\leq r\leq a$, or (b) when $S$ lies in one of the wings shown in Lemma \ref{OmegaConfig} (ii) and (iii).

\underline{Case (a)}:
It follows from Lemma \ref{S_t-hom} that $\sthom_A(\Omega(X_i(mn+j_t)), S_r)$ is one-dimensional if $r=t+1$ and $t<a$; zero, otherwise.

\underline{Case (b)}:
By Lemma \ref{hi-hom}, $\sthom_A(\Omega(X_i(mn+j_t)), S)=0$ if the quasi-simple $M:=X_{\overline{i+j_t-1}}$ lies in $\mathcal{W}_{\overline{i+j},s}$, so it suffices to show that this is impossible.
Indeed, the quasi-simples not contained by the union the wings stated in Lemma \ref{OmegaConfig} (ii) and (iii) are precisely $X_{\overline{i+j_t-1}}$ for all $1\leq t\leq a$.
The claim now follows.
\end{Proof}

We have gathered all ingredient to prove Theorem \ref{main-result-2}.

\begin{Proof}[Proof of Theorem \ref{main-result-2}]
We will assume the setting of Assumption \ref{assump}, and aim to prove that the sequence $(M_l)_{l \geq 0}$ satisfies the conditions of Proposition \ref{main-strat}.  In particular, the claim of the theorem is just an immediate consequence of Proposition \ref{main-strat}.

\underline{Condition (i)}:  For any $l\geq 0$, Lemma \ref{OmegaHom} (i) and the orthogonality of $\mathcal{S}$ says that $M_l$ can be in $\mathcal{S}$ unless it is isomorphic to $X_i(n)$.
The difference in quasi-lengths implies that $M_l\ncong X_i(n)$ for all $l>0$ and Lemma \ref{omega-fix-C} says that $M_0=\Omega(X_i(n))\ncong X_i(n)$, too.

\underline{Condition (ii)}:
By Lemma \ref{OmegaHom}, a triangle of the form $N\to \Omega(X_i(mn+j_t))\to S\to $ is non-split only when $S=X_i(n)$ or $S=\Omega(X_i(j_{t+1}))$, where the later case does not appear when $t=a$.
Moreover, these triangles are unique up to isomorphism as $\sthom_A(\Omega(X_i(mn+j_t)), S)\cong k$.
Hence, they must be the ones given in Lemma \ref{sms-triangles}.
Now it follows from the description of these triangles that we have $N=M_{l'}$ with $l'>l$, as required.
\end{Proof}


\bigskip
\begin{thebibliography}{88}

\bibitem{ARS}{{\sc
M.Auslander, I.Reiten and S.O.Smal\o,} {\it Representation theory of
Artin algebras}. Cambridge University Press, 1995.}

\bibitem{BS}{{\sc J.Bialkowski and A.Skowronski,} Selfinjective algebras of tubular type. Colloquium
Mathematicum \textbf{94} (2) (2002), 175-194.}

\bibitem{CKL}{{\sc A.Chan, S.Koenig and Y.Liu,} Simple-minded systems, configurations and
		mutations for representation-finite self-injective algebras. J. Pure Appl. Algebra \textbf{219} (2015), 1940-1961.}


\bibitem{CB}{{\sc W.W.Crawley-Boevey,} On tame algebras and bocses. Proc. London Math. Soc. \textbf{56} (1988), 451-483.}

\bibitem{Dugas}{{\sc A.Dugas,} Torsion pairs and simple-minded systems in
		triangulated categories. Appl. Categ. Structures {\bf 23}  (2015), 507-526.}


\bibitem{Erdmann}{{\sc K.Erdmann,} {\it Blocks of tame representation type and related algebras}. Lecture Notes in Mathematics 1428, Springer, Berlin, 1990.}

\bibitem{EK}{{\sc K.Erdmann and O.Kerner,} On the
stable module category of a self-injective algebra. Trans. Amer.
Math. Soc. \textbf{352} (2006), 2389-2405.}

\bibitem{Happel}{{\sc D.Happel,} {\it
Triangulated categories in the representation theory of finite
dimensional algebras}. London Math. Soc. Lecture Notes vol. 119,
Cambridge University Press. 1988.}


\bibitem{Hoshino}{{\sc M.Hoshino,} {$D{\rm Tr}$}-invariant modules.
Tsukuba J. Math. \textbf{7} (2) (1983), 205-214.}

\bibitem{KL}{{\sc S.Koenig and
Y.Liu,} Simple-minded systems in stable module categories. Quart. J.
Math. \textbf{63} (3) (2012), 653-674.}

\bibitem{Pogorzaly1994}{{\sc Z.Pogorza\l
y,} Algebras stably equivalent to self-injective special biserial
algebras. Comm. in Algebra \textbf{22}(4) (1994), 1127-1160.}

\bibitem{MS}{{\sc P.Malicki and A.Skowronski,} On the number of simple and projective modules in the quasi-tubes of self-injective algebras. Comm. in Algebra \textbf{39} (2011), 322-334.}

\bibitem{Ringel}{{\sc C.M.Ringel,} {\it Tame algebras and integral quadratic forms}. Lecture Notes in Mathematics 1099, Springer, Berlin, 1984.}

\end{thebibliography}
\end{document}